\documentclass[11pt, reqno]{amsart}
\usepackage[T1]{fontenc}
\usepackage{amsfonts}
\usepackage{amsmath}
\usepackage{amsthm}
\usepackage{amssymb}
\usepackage{geometry}
\usepackage{graphicx}
\usepackage{xcolor}
\usepackage{mathtools}
\usepackage[colorlinks=true, linkcolor=blue]{hyperref}
\usepackage{cleveref}
\usepackage[english]{babel}
\usepackage{lineno}
\usepackage{float}

\newtheorem{theorem}{Theorem}[section]

\newtheorem{conjecture}[theorem]{Conjecture}
\newtheorem{proposition}[theorem]{Proposition}

\usepackage{tikz}
\usetikzlibrary{positioning}
\usetikzlibrary{decorations,arrows}
\usetikzlibrary{decorations.markings}
\numberwithin{equation}{section}

\newtheorem{claim}{}[theorem]

\newcommand{\Z}{\mathbb Z}
\newcommand{\R}{\mathbb R}

\newcommand{\X}{\Omega}

\usepackage{rustic}
\usepackage[T1]{fontenc}

\newcommand{\und}{\;\mbox{ and }\;}
\newcommand{\nn}{\nonumber}

\newcommand{\be}{\begin{equation}}
\newcommand{\ee}{\end{equation}}

\emergencystretch=\maxdimen
\title{Pollard's theorem in general abelian groups}

\author{David J.~Grynkiewicz}
\address[D.~J.~Grynkiewicz]{Department of Mathematical Sciences, University of Memphis, Memphis, TN 38152, USA}
\email{diambri@hotmail.com}
\thanks{}

\author{Runze Wang}
\address[R.~Wang]{Department of Mathematical Sciences, University of Memphis, Memphis, TN 38152, USA}
\email{runze.w@hotmail.com}
\thanks{}


\begin{document}

\sloppy

\begin{abstract}
    We make further progress towards a Kneser-type generalization of Pollard's Theorem to general abelian groups. For two sets $A$ and $B$ in an abelian group $G$, the \emph{$t$-popular sumset} of $A$ and $B$, denoted by $A+_t B$, is the set of elements in $G$ each with at least $t$ representations of the form $a+b$, where $a\in A$ and $b\in B$. For $|A|,\, |B|\ge t\geq 2$, we prove that if
    \begin{align*}
        \sum_{i=1}^t |A+_i B|< t|A|+t|B|-\frac{4}{3}t^2+\frac{2}{3}t,
    \end{align*}
    then there exist $A'\subseteq A$ and $B'\subseteq B$ with $|A\setminus A'|+|B\setminus B'|\le t-1$, $A'+_t B'=A'+B'=A+_t B$, and $ \sum_{i=1}^t |A+_i B|\ge t|A|+t|B|-t|H|,$ where $H$ is the stabilizer of $A'+B'=A+_t B$. Our result improves the main quadratic term in the previous best bound from $-2t^2$ to $-\frac{4}{3}t^2$.
\end{abstract}

\maketitle

\section{Introduction}

\subsection{Basic Notation and Concepts}

Let $(G,\ +)$ be an abelian group, and let $A,\, B\subseteq G$ be subsets. The \emph{sumset} of $A$ and $B$ is defined by
\begin{align*}
    A+B:=\{a+b: a\in A,\ b\in B\}.
\end{align*}
Similarly, for an element $g\in G$, \ $g+A:=\{g+a: a\in A\}$.
We use
\begin{align*}
\mathsf  r_{A,\,B}(g):=|\{(a,\ b)\in A\times B: a+b=g\}|
\end{align*}
to count the number of representations of $g$ of the form $a+b$, with $a\in A$ and $b\in B$. It is easily noted  that $\mathsf r_{A,\,B}(g)=|(g-A)\cap B|=|(g-B)\cap A|$.
Let $t$ be a positive integer. The \emph{$t$-popular sumset} of $A$ and $B$ is defined to be
\begin{align*}
    A+_t B:=\{g\in G:\;\mathsf  r_{A,\,B}(g)\ge t\}.
\end{align*}
In particular, if we take $t=1$, then $A+_1 B$ is just $A+B$.
The \emph{stabilizer} of $A$ is defined to be
\begin{align*}
    \mathsf{H}(A):=\{e\in G: e+A=A\}\leq G,
\end{align*}
which is easily seen to be a subgroup of $G$. If $|\mathsf{H}(A)|\ge 2$, then $A$ is said to be \emph{periodic}; otherwise $A$ is said to be \emph{aperiodic}. If $H\leq \mathsf H(A)$, that is, $A$ is a union of $H$-cosets, then $A$ is said to be \emph{$H$-periodic}.
For $a,\, b\in \Z$, let $[a,\, b]:=\{x\in \Z: \; a\leq x\leq b\}$ denote a discrete interval.

The following basic lower bound on $\mathsf r_{A,\,B}(g)$ will be referred to as the \emph{Pigeonhole Bound} (as it is an immediate consequence of the Pigeonhole Principle), and it will be used multiple times in the proof of the main result (Theorem \ref{new}) in this paper.

\begin{proposition}[Pigeonhole Bound {\cite[p.~57]{Gry2}}] \label{pigeon}
    Let $G$ be an abelian group, and let $A,\, B\subseteq G$ be finite subsets. Then $
       \mathsf  r_{A,\,B}(g)\ge |A|+|B|-|G|$ for any $g\in G$.
\end{proposition}

\subsection{Classical Theorems}

The well-known Cauchy-Davenport Theorem  gives a lower bound on the size of $A+B$ when the group $G$ is taken to be a prime order group $C_p:=\Z/p\Z$, and is among the oldest results in this area of Combinatorial Number Theory, dating to 1813.

\begin{theorem}[Cauchy-Davenport Theorem  \cite{Cau, Dav, Gry2}]
    Let $p$ be a prime number, let $G=C_p$, and let $A,\, B\subseteq G$ be subsets with $|A|,\,|B|\geq 1$. Then
    \begin{align*}
        |A+B|\ge \min\{p,\ |A|+|B|-1\}.
    \end{align*}
\end{theorem}

Kneser's Theorem  generalizes the Cauchy-Davenport Theorem to general abelian groups.

\begin{theorem}[Kneser's Theorem \cite{Kne2,Gry2}]\label{kneserthm}
    Let $G$ be an abelian group, and let $A,\, B\subseteq G$ be finite subsets with $|A|,\,|B|\geq 1$. Then
    \begin{align}\label{kt-bound}
        |A+B|\ge |A+H|+|B+H|-|H|,
    \end{align}
    where $H:=\mathsf{H}(A+B)$.
\end{theorem}

 Kneser's Theorem can be equivalently stated with the bound \eqref{kt-bound} replaced by the bound $|A+B|\geq |A|+|B|-|H|$, as a short, couple lines argument derives \eqref{kt-bound} from this seemingly weaker bound (see \cite{Gry2}).
Kneser's Theorem also immediately implies (via Proposition \ref{pigeon}) the following lower bound  on $\mathsf r_{A,\,B}(g)$, which predates Kneser's Theorem.

\begin{proposition}[{\cite[p.~57]{Gry2}}] \label{multi}
    Let $G$ be an abelian group, and let $A,\, B\subseteq G$ be finite subsets with $|A|,\,|B|\geq 1$. Then $\mathsf r_{A,B}(g)\geq |A|+|B|-|A+B|$ for any $g\in A+B$.
\end{proposition}

The Cauchy-Davenport Theorem involves ordinary sumsets in prime order groups.  Kneser's Theorem generalizes the Cauchy-Davenport Theorem giving a result valid for general abelian groups rather than prime order groups. Pollard generalized the Cauchy-Davenport Theorem giving a result for $t$-popular sums rather than ordinary sumsets.

\begin{theorem}[Pollard's Theorem \cite{Pol, Gry2}]
    Let $t$ be a positive integer, let $p$ be a prime number, let $G=C_p$, and let $A,\, B\subseteq G$ be subsets with $|A|,\,|B|\geq t$. Then
    \begin{align}\label{pollard-bound}
        \sum_{i=1}^t|A+_i B|\ge  \min\{tp,\ t|A|+t|B|-t^2\}.
    \end{align}
\end{theorem}

An equivalent form of Pollard's Theorem is obtained  by removing the hypothesis $|A|,\,|B|\geq t$ and adding $|A||B|$ to the quantities minimized in \eqref{pollard-bound}.

\subsection{Generalizing Pollard's Theorem to General Abelian Groups}

It is natural to consider combining Kneser's Theorem and Pollard's Theorem, i.e., finding a common generalization of both  results.
Kneser's Theorem can be viewed as a partial structural description of subsets with very small sumset: If $|A+B|<|A|+|B|-1$, then $|A+B|=|A|+|B|-|H|+\rho$, where $\rho=|(A+H)\setminus A|+|(B+H)\setminus B|$, meaning
$A$ and $B$ are ``large'' subsets of the $H$-periodic subsets $A+H$ and $B+H$, where large is quantified as being contained in $A+H$ and $B+H$ with at most $\rho=|A+B|-|A|-|B|+|H|\leq |H|-2$ missing elements. The complete if-and-only-if characterization of which subsets $A$ and $B$ (of a general abelian group $G$) satisfy $|A+B|\leq |A|+|B|-1$, known as the Kemperman Structure Theorem,  was later accomplished by Kemperman \cite{Kemp,Gry2}. Thus a full generalization of both Kneser and Pollard's Theorems should provide a weak  structural description of those subsets $A$ and $B$ (of a general abelian group $G$) with
$\sum_{i=1}^{t}|A+_iB|<t|A|+t|B|-t^2$ that, in the case $t=1$, readily yields Kneser's Theorem.

 Regarding partial work along these lines,  we have the following result due to Hamidoune and Serra \cite{Gry2,HS}, which generalizes a lemma of  Green and Ruzsa needed by them for counting the number of sum-free sets \cite{Gre-Ruz}.

\begin{theorem}[Hamidoune and Serra \cite{HS,Gry2}]\label{ham-serr-thm}
    Let $t$ be a positive integer, let $G$ be an abelian group, and let $A,\, B\subseteq G$ be finite subsets with $|A|,\, |B|\ge t$. Let $H\le G$ be a subgroup of maximal cardinality such that $A+B$ contains $g+H$ for some $g\in G$. Then
    \begin{align*}
        \sum_{i=1}^t |A+_i B|\ge t|A|+t|B|-t^2-\frac{1}{4}|H|^2.
    \end{align*}
\end{theorem}

This theorem implies that, if $\sum_{i=1}^t |A+_i B|<t|A|+t|B|-t^2$, then there is a nontrivial subgroup $H$ such that $A+B$ contains $g+H$ for some $g\in G$. The bound, when so formulated, is the same as the one in Pollard's Theorem, but, compared to  Kneser's Theorem, the structural information is not nearly as strong as it only gives  structure for a small local portion of $A+B$.

The result of Hamidoune and Serra is valid for sumsets $A+B$ satisfying the target hypothesis $\sum_{i=1}^t |A+_i B|<t|A|+t|B|-t^2$, but only provides weak structural information. On the other hand, Grynkiewicz proved a result \cite{Gry1,Gry2} which provides strong structural information, but which requires a more stringent upper bound hypothesis for $\sum_{i=1}^t |A+_i B|$.

\begin{theorem}[Grynkiewicz \cite{Gry2}] \label{old}
    Let $t$ be a positive integer, let $G$ be an abelian group, and let $A,\, B\subseteq G$ be finite subsets with $|A|,\, |B|\ge t$. If
    \begin{align}
        \sum_{i=1}^t |A+_i B|< t|A|+t|B|-2t^2+3t-2, \label{eq1}
    \end{align}
    then there exist subsets $A'\subseteq A$ and $B'\subseteq B$ such that
    \begin{align}
        &|A\setminus A'|+|B\setminus B'|\le t-1 \quad\und \quad\label{eq2}A'+_t B'=A'+B'=A+_t B.
       \end{align}
  Moreover, $|A'+B'|<|A'|+|B'|-1$ and
  \be
  \sum_{i=1}^t |A+_i B|\ge t|A|+t|B|-t^2-(t-\ell)(|H|-\rho-t)\ge t|A|+t|B|-t|H|, \label{eq4}
    \ee
    where $H=\mathsf{H}(A+_t B)$, \ $\ell=|A\setminus A'|+|B\setminus B'|$, and $\rho=|(A'+H)\setminus A'|+|(B'+H)\setminus B'|$.
\end{theorem}

The structural conclusions in Theorem \ref{old} are fairly strong.
The set of $t$-popular sums $A+_tB$ is forced to  equal an ordinary sumset $A'+B'$, and the conclusion $|A'+B|<|A'|+|B'|-1$ means the precise possibilities for the structure of $A'$ and $B'$ are  given by the Kemperman Structure Theorem mentioned earlier. The sets $A$ and $B$ are then obtained by taking the highly structured sets $A'$ and $B'$ and adding a small number of additional elements, at most $t-1$,  and in the case $t=1$, the bound \eqref{eq4} becomes an equivalent statement of Kneser's Theorem.
We  remark that the conclusions $|A'+B'|<|A'|+|B'|-1$ and \eqref{eq4} are  implied by \eqref{eq2} (assuming $A'$ and $B'$ chosen extremely), though this requires a moderate amount of work (see Proposition \ref{mainprop}).
Theorem \ref{old} has also been used as a needed tool for various applications \cite{CK,GL,MT,Pla,Pok}.

Our main result  is the following strengthening of Theorem \ref{old},  showing  the same structural conclusions hold under the weaker hypothesis $\sum_{i=1}^t |A+_i B|< t|A|+t|B|-\frac{4}{3}t^2+\frac{2}{3}t$. Thus we improve the main quadratic term in \eqref{eq1} from $-2t^2$ to $-\frac43t^2$, closer to the target goal of $-t^2$ needed for a full generalization of both Pollard's and Kneser's Theorems. Note,  the case $t=1$ could be included in Theorem \ref{new} by adding the constant $-\frac13$ to the bound in \eqref{hypothesis}. In the case $t=2$, we have $-2t^2+3t-2=\lceil-\frac{4}{3}t^2+\frac{2}{3}t\rceil=-t^2$, so Pollard's Theorem is fully generalized for $t\leq 2$, though not for $t\geq 3$ as there is then a gap between the bounds.  However, as noted in Section \ref{sec-conj}, the bound \eqref{hypothesis} is nonetheless optimal for $t\in [2,4]$.
We also give  a short application of our theorem to restricted sumsets later in Section \ref{sec-app}.

\begin{theorem} \label{new}
    Let $t\geq 2$ be a positive integer, let $G$ be an abelian group, and let $A,\, B\subseteq G$ be finite subsets with $|A|,\, |B|\ge t$. If
    \begin{align}
        \sum_{i=1}^t |A+_i B|< t|A|+t|B|+\biggl\lceil-\frac{4}{3}t^2+\frac{2}{3}t\biggr\rceil, \label{hypothesis}
    \end{align}
    then
    there exist subsets $A'\subseteq A$ and $B'\subseteq B$ such that
    \begin{align}
        &|A\setminus A'|+|B\setminus B'|\le t-1 \quad\und\quad \label{conclusion}A'+_t B'=A'+B'=A+_t B.
       \end{align}
  Moreover, $|A'|,\,|B'|\geq t+1$, \  $|A'+B'|<|A'|+|B'|-t$, and
  \be
  \sum_{i=1}^t |A+_i B|\ge t|A|+t|B|-t^2-(t-\ell)(|H|-\rho-t)\ge t|A|+t|B|-t|H|, \label{conclusion-impliedbound}
    \ee
    where $H=\mathsf{H}(A+_t B)$, \ $\ell=|A\setminus A'|+|B\setminus B'|$, and $\rho=|(A'+H)\setminus A'|+|(B'+H)\setminus B'|$.

\end{theorem}

\section{Proof of the Main Result}

We will need the following result from \cite[Proposition 12.1]{Gry2}, which we state in slightly fuller generality than it appears in \cite{Gry2}, though we will only use the case $\alpha=0$. The proof of Proposition \ref{mainprop} (as stated below) is identical to that of \cite[Proposition 12.1]{Gry2}, apart from modifying the straightforward estimates to take into account the factor $\alpha$, the improved bound given in Item 3 (which was stated in the body of the original proof from  \cite[Proposition 12.1]{Gry2} as (12.14)), and the improved bound in Item 2, which follows from combining \eqref{alphabound} with Items  3 and 5. It is also a fairly routine application of Kneser's Theorem.  Of particular importance, if \eqref{conclusion}  holds for some subsets $A'\subseteq A$ and $B'\subseteq B$, then replacing $A'$ and $B'$ by the sets $A''$ and $B''$ defined in  Proposition \ref{mainprop} results in both \eqref{conclusion} and \eqref{conclusion-impliedbound} holding, with various other useful properties (including that $|A'+B'|<|A'|+|B'|-t$ and $|A'|,\,|B'|\geq |H|-\rho\geq t+1$) then also following from Proposition \ref{mainprop}.

\begin{proposition}[Grynkiewicz \cite{Gry2}]\label{mainprop}
    Let $t$ be a positive integer, let $\alpha\geq 0$ be a nonnegative real number, let $G$ be an abelian group, and let $A,\, B\subseteq G$ be finite subsets with $|A|,\,|B|\geq t$ and
    \begin{align}
        \sum_{i=1}^t |A+_i B|&<t|A|+t|B|-(1+\alpha)t^2. \label{alphabound}
    \end{align}
    Assume \eqref{conclusion} holds using the subsets $A'\subseteq A$ and $B'\subseteq B$. Let $H=\mathsf{H}(A+_t B)$, \ $A''=(A'+H)\cap A$, $B''=(B'+H)\cap B$, $\rho=|(A''+H)\setminus A''|+|(B''+H)\setminus B''|$, and $\ell=|A\setminus A''|+|B\setminus B''|$. Then the following all hold.
    \begin{enumerate}
        \item[1.] \eqref{conclusion} and \eqref{conclusion-impliedbound} hold using the subsets $A''\subseteq A$ and $B''\subseteq B$.
        \item[2.] $|A''+B''|<|A''|+|B''|-(1+\alpha)t$.
        \item[3.]  $\sum_{i=1}^t|A+_iB|\geq \sum_{i=1}^t |A''+_i B''|+\ell(|H|-\rho)\geq t|A|+t|B|-t^2-(t-\ell)(|H|-\rho-t)$.
        \item[4.] For any $a\in A\setminus A''$, there exists $b''\in B''$ such that $\Bigl(a+\bigl((b''+H)\cap B''\bigr)\Bigr)\cap (A+_t B)=\emptyset$, which further implies that $|(a+B'')\setminus (A+_t B)|\ge |H|-\rho$; and for any $b\in B\setminus B''$, there exists $a''\in A''$ such that $\Bigl(\bigl((a''+H)\cap A''\bigr)+b\Bigr)\cap (A+_t B)=\emptyset$, which further implies that $|(A''+b)\setminus (A+_t B)|\ge |H|-\rho$.
        \item[5.] $|H|-\rho\ge \big\lfloor\frac{(1+\alpha)t^2-t\ell}{t-\ell} \big\rfloor+1\geq \lfloor (1+\alpha)t\rfloor+1>(1+\alpha)t$.
        \item[6.] If $\ell=t-1$, then
        \begin{align*}
            &\sum_{i=1}^t |(A\cup\{g\})+_i B|\ge t|A\cup\{g\}|+t|B|-t^2\quad\und\\
            &\sum_{i=1}^t |A+_i (B\cup\{h\})|\ge t|A|+t|B\cup\{h\}|-t^2
        \end{align*}
        for any $g\in G\setminus (A''+H)$ and any $h\in G\setminus (B''+H)$.
    \end{enumerate}
\end{proposition}

Our proof of Theorem \ref{new} builds upon the framework established in the proof of Theorem \ref{old}, but requires some  key enhancements. The core of both proofs is an inductive argument built around the Dyson Transform, which replaces the sets $A$ and $B$ by $A\cup B$ and $A\cap B$ (after appropriate extremal translations of $A$ and $B$).  Compared to the proof of Theorem \ref{old}, our proof of Theorem \ref{new} incorporates a $4$-fold (rather than $3$-fold) induction, adding induction along the parameter  $t$. Besides this, there are three points in the proof of Theorem \ref{old} where the stronger hypothesis $\sum_{i=1}^t|A+_iB|<t|A|+t|B|-2t^2+3t-2$ is needed rather than only the desired hypothesis $\sum_{i=1}^t|A+_iB|<t|A|+t|B|-t^2$. The original arguments proving Theorem \ref{old} break down at these points, requiring new and more involved arguments to circumvent these difficulties. For readers wishing to compare with the proof of Theorem \ref{old}, we mention beforehand that, in the following proof of Theorem \ref{new}, these three points are:
\begin{enumerate}
    \item \ref{claimC}: $|B|\ge 2t-1$.
    \item \ref{claimD}: $|B(0)|\ge t$.
    \item Case 2: $A'=A(0)$, but there exists $b\in B(0)\setminus B'$.
\end{enumerate}

Essentially, \ref{claimB}, Case 1, and Case 3 in our proof rely on the same logic as Claim A, Subcase 1.1, and Case 2 in the proof of Theorem \ref{old}  \cite[pp.~161--174]{Gry2}. For completeness, we have included the concise proofs of \ref{claimB} and Case 1 in the  main proof below. However, in the interest of avoiding too much redundancy and saving space, we have omitted the details for the longer and more involved Case 3. The argument is  identical to the corresponding part in \cite[Case 2 in Theorem 1.7]{Gry2} and is rather self-contained, so its omission does not interfere with the flow of ideas in other parts.

\begin{proof}[Proof of Theorem \ref{new}]
The proof is built upon a quadruple induction on
\begin{align*}
    \Biggl(t,\ \sum_{i=1}^t |A+_i B|,\ -(|A|+|B|),\ \min\{|A|,\, |B|\}\Biggr)
\end{align*}
using the lexicographic order. For $t=2$,
we have $-2t^2+3t-2=\lceil-\frac{4}{3}t^2+\frac{2}{3}t\rceil=-t^2$, so the hypotheses \eqref{eq2} and \eqref{hypothesis} coincide. In such case, Theorem \ref{old} (and Proposition \ref{mainprop}) yield all desired conclusions. Thus we can assume $t\geq 3$ with the base of the induction complete, and we now assume that, for any pair of $\mathcal{A}$ and $\mathcal{B}$ which appears before $A$ and $B$ in the lexicographic order, Theorem \ref{new} is true. From  \eqref{hypothesis} and $t\geq 2$, we have
\begin{align}
    \sum_{i=1}^t |A+_i B|
    &\leq t|A|+t|B|-\frac{4}{3}t^2+\frac{2}{3}t-\frac{1}{3} \label{tbound} \\
    &< t|A|+t|B|-t^2, \label{tboundaux}
\end{align}
By symmetry, we may assume $|A|\ge |B|$. We can see that if $|B|=t$, then no element in $A+B$ has more than $t$ representations, so
    $$\sum_{i=1}^t |A+_i B|=|A||B|=t|A|+t|B|-t^2,$$
contrary to \eqref{tboundaux}. Thus $$|A|\ge |B|\ge t+1.$$
In view of Proposition \ref{mainprop}, we need to show \eqref{conclusion} holds, which we now assume by contradiction fails.

\begin{claim}\label{claimA} $\sum_{i=1}^{t-1}|A+_iB|\geq(t-1)|A|+(t-1)|B|-\frac43(t-1)^2+\frac23(t-1)$.
\end{claim}

\begin{proof}[Proof of \ref{claimA}]
Assume by contradiction that
\begin{align*}
    \sum_{i=1}^{t-1} |A+_i B|
    &\leq (t-1)|A|+(t-1)|B|-\frac{4}{3}(t-1)^2+\frac{2}{3}(t-1)-\frac{1}{3}\\
    &< (t-1)|A|+(t-1)|B|-(t-1)^2,
\end{align*}
with the latter inequality in view of $t\geq 3$.
Then, by the induction hypothesis, there exist $A'\subseteq A$ and $B'\subseteq B$ such that
\begin{align}
    &\ell:=|A\setminus A'|+|B\setminus B'|\le (t-1)-1\; \nn\und\;A'+_{t-1} B'=A'+B'=A+_{t-1} B.
\end{align}
Let $H=\mathsf{H}(A+_{t-1} B)$ and $\rho=|(A'+H)\setminus A'|+|(B'+H)\setminus B'|$.

We apply Proposition \ref{mainprop} to $t-1$, $A$, and $B$, and replace $A'$ and $B'$ by the resulting $A''$ and $B''$. By Item 5 in Proposition \ref{mainprop}, we have $|H|-\rho\ge (t-1)+1=t$. For an $H$-coset $g+H$, $(g+H)\cap A'$ is called an \emph{$H$-coset slice} in $A'$, and $(g+H)\cap B'$ is called an \emph{$H$-coset slice} in $B'$. Applying the Pigeonhole Bound  (Proposition \ref{pigeon}) to each pair of $H$-coset slices in $A'$ and $B'$, we know that every element in $A'+B'=A'+_{t-1} B'=A+_{t-1} B$ actually has at least $|H|-\rho\ge t$ representations, so
    $A'+B'=A'+_t B'=A+_t B$. But now \eqref{conclusion} holds, contrary to assumption.
\end{proof}

\begin{claim}\label{claimB}For every $x\in A$ and $y\in B$, we have
    \begin{align}
        |(x+B)\setminus (A+_{t+1}B)|\le t-1\;\und\;
        |(A+y)\setminus (A+_{t+1}B)|\le t-1. \nn
    \end{align}
    \end{claim}

\begin{proof}[Proof of \ref{claimB}]

    First, assume by contradiction that
    \begin{align}
        |(x+B)\setminus (A+_{t+1}B)|\geq t \label{claima1}
    \end{align} for some $x\in A$. Note $ |(x+B)\setminus (A+_{t+1}B)|$
    counts the number of elements in $x+B$ each having at most $t$ representations in $A+B$. So \eqref{claima1} implies that there are at least $t$ such elements.

    We have $|A\setminus \{x\}|,\, |B|\ge t$ because $|A|,\, |B|\ge t+1$, while $\sum_{i=1}^t |(A\setminus\{x\})+_i B|<\sum_{i=1}^t |A+_i B|$.
    We also have
    \begin{align*}
        \sum_{i=1}^t |(A\setminus \{x\})+_i B|< t|A\setminus\{x\}|+t|B|+\biggl\lceil
        -\frac{4}{3}t^2+\frac{2}{3}t\biggr\rceil.
    \end{align*}
    Otherwise,  by \eqref{claima1}, we have
        $\sum_{i=1}^t |A+_i B|\ge t|A\setminus\{x\}|+t|B|+\lceil
        -\frac{4}{3}t^2+\frac{2}{3}t\rceil+t
        =t|A|+t|B|+\lceil-\frac{4}{3}t^2+\frac{2}{3}t\rceil,$
    which contradicts \eqref{tbound}. Thus the induction hypothesis and Proposition \ref{mainprop} can be applied to $A\setminus \{x\}$ and $B$. Accordingly, \be\label{tagyourit}A'+B'=A'+_{t}B'=(A\setminus \{x\})+_tB\ee for some $A'\subseteq A\setminus \{x\}$ and $B'\subseteq B$ with
        \be\label{peekaboo}\ell:=|(A\setminus \{x\})\setminus A'|+|B\setminus B'|\le t-1.\ee
     Let
        $H=\mathsf{H}\bigl((A\setminus \{x\})+_t B\bigr)$. Applying Item 1 in Proposition \ref{mainprop} to $A\setminus \{x\}+B$ and replacing $A'$ and $B'$ by the resulting $A''$ and $B''$, we can w.l.o.g.~assume that $A'=(A'+H)\cap (A\setminus \{x\})$ and $B'=(B'+H)\cap B$.

Suppose $x\in A'+H$. Then by \eqref{tagyourit}, we have
    \begin{align*}
        x+B'\subseteq A'+H+B'=A'+B'=(A\setminus \{x\})+_t B,
    \end{align*}
    which means every element in $x+B'$ has at least $t$ representations in $(A\setminus \{x\})+B$, and thus each of them has at least $t+1$ representations in $A+B$. Hence, in $x+B$, there are at most $|x+(B\setminus B')|=|B\setminus B'|\le \ell\le t-1$ elements, each with at most $t$ representations in $A+B$, contrary to our assumption \eqref{claima1}. So we must instead  have $x\notin A'+H$ and $A'=(A'+H)\cap (A\setminus \{x\})=(A'+H)\cap A$.

     If $\ell=t-1$,  then we can apply Item 6 in Proposition \ref{mainprop} to  $A\setminus \{x\}$ and $B$ using $g=x\notin A'+H$, whence
        $\sum_{i=1}^t |A+_i B|=\sum_{i=1}^t \bigl|\bigl((A\setminus\{x\})\cup\{x\}\bigr)+_i B\bigr|
        \ge t\bigl|(A\setminus\{x\})\cup\{x\}\bigr|+t|B|-t^2
        =t|A|+t|B|-t^2,$
    which contradicts \eqref{tboundaux}. Therefore $\ell\le t-2$.
We will
 show that, using the same $A'$ and $B'$, \eqref{conclusion} holds for $A$ and $B$, contradicting that we have assumed this failed by assumption.

    By \eqref{peekaboo} and $\ell\le t-2$, we have
    \begin{align*}
        |A\setminus A'|+|B\setminus B'|=|(A\setminus\{x\})\setminus A'|+1+|B\setminus B'|=\ell+1\le t-1,
    \end{align*}
    so the first inequality from \eqref{conclusion} holds for $A$ and $B$.

    By \eqref{tagyourit}, any  $e\in(A\setminus \{x\})+B$ with  $\mathsf r_{A\setminus \{x\},\,B}(e)\le t-1$ has $e\notin A'+B'$. So by \eqref{peekaboo}, we have $\mathsf r_{A\setminus \{x\},\,B}(e)\le \ell$. Thus
    \begin{align*}
        (A\setminus \{x\})+_t B=(A\setminus \{x\})+_{\ell+1} B,
    \end{align*}
    which further implies $(A\setminus \{x\})+_t B=(A\setminus \{x\})+_{t-1} B$ as $\ell\le t-2$. So, now we have
    \begin{align*}
        A'+B'=A'+_t B'=(A\setminus \{x\})+_t B=(A\setminus \{x\})+_{t-1} B.
    \end{align*}
    But we also have $A+_t B\subseteq (A\setminus \{x\})+_{t-1} B$, so $A+_t B\subseteq A'+_t B'$, and thus $A'+B'=A'+_t B'=A+_t B$, which means \eqref{conclusion} holds for $A$ and $B$, contrary to assumption. This shows that  $|(x+B)\setminus (A+_{t+1}B)|\le t-1$. The assumption that $|A|\ge |B|$ is not used in this proof, so by symmetry, the same argument also shows $|(A+y)\setminus (A+_{t+1}B)|\le t-1$ for any $y\in B$.
\end{proof}

We construct a dot grid with $|A+B|$ rows and $|B|$ columns, where each row corresponds to an element in $A+B$, and on the row corresponding to $x\in A+B$, we place a dot on column $i$ for each $i\in [1,\, \mathsf r_{A,\,B}(x)]$. Then, for any $x\in A+B$, the number of dots on the row corresponding to $x$ is the number of representations of $x$ in the sumset $A+B$; and for any $j\in [1,\, |B|]$, the number of dots on the $j$-th column is equal to $|A+_j B|$. Up to reordering the elements in $A+B$, we may assume that the row lengths are sorted in descending order, with the first row being the longest. An example with $|B|=6$ and $|A+B|=13$ is shown in Figure \ref{ex1}.

\begin{figure}[H]
    \centering
    \input ex1.tex
\end{figure}

We draw a line (the blue line in Figure \ref{ex2}) between the $t$-th and the $(t+1)$-th column. Then $\sum_{i=1}^t |A+_i B|$ counts the number of dots to the left of this line. Let
\begin{align*}
    X:= A+_{t+1} B.
\end{align*}
Then for an element $x\in A+B$, we have $x\in X$ if and only if the length of the row corresponding to $x$ is at least $t+1$. There are at most $|X|(|B|-t)$ dots to the right of this line and thus not counted by $\sum_{i=1}^t |A+_i B|$. In this dot grid, if there is not a dot at $(i,\ j)$, then we say there is a \emph{hole} at $(i,\ j)$. Let
\begin{align*}
    y:=|X||B|-\sum_{x\in X}\mathsf r_{A,\,B}(x)
\end{align*}
be the number of holes in the $|X|\times (|B|-t)$ rectangle from row $1$ to row $|X|$ and from column $t+1$ to column $|B|$, which is illustrated by the red dashed rectangle in Figure \ref{ex2}.

\begin{figure}[H]
    \centering
    \input ex2.tex
\end{figure}

There are $|A||B|$ total dots, and exactly $|X|(|B|-t)-y$ dots not counted by $\sum_{i=1}^t |A+_i B|$, so
\begin{align}\label{holes}
    \sum_{i=1}^t |A+_i B|=|A||B|-|X|(|B|-t)+y.
\end{align}

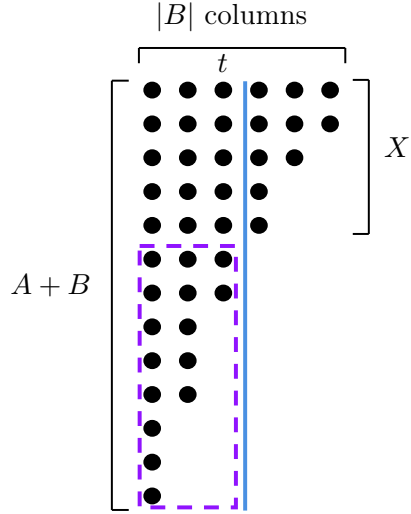
We define a bipartite graph $\Gamma_t=\Gamma_t(A,\ B)$ with vertex partite classes $A$ and $B$ and an edge between $a\in A$ and $b\in B$ when $a+b\notin X$. Let $E:=\{(a,\ b): a\in A,\ b\in B,\ a+b\notin X\}$ be the edge set of $\Gamma_t$. Then $|E|$ is the number of dots in the $(|A+B|-|X|)\times t$ rectangle from row $|X|+1$ to row $|A+B|$ and from column $1$ to column $t$, which is illustrated by the purple dashed rectangle in Figure \ref{ex3}. So, we have
\begin{align}\label{tXE}
    \sum_{i=1}^t |A+_i B|=t|X|+|E|.
\end{align}

\begin{figure}[H]
    \centering
    \input ex3.tex
\end{figure}

By \ref{claimB}, we have
\begin{align}\label{claima111222}
    |(a+B)\cap X|\ge |B|-t+1 \quad \text{and} \quad |(A+b)\cap X|\ge |A|-t+1
\end{align}
for any $a\in A$ and $b\in B$.

Since
\begin{align*}
    \sum_{i=1}^{t-1} |A+_i B|=\sum_{i=1}^t |A+_i B|-|A+_t B|\le \sum_{i=1}^t |A+_i B|-|A+_{t+1} B|=\sum_{i=1}^t |A+_i B|-|X|,
\end{align*}
combining \ref{claimA} and \eqref{tbound} yields
\begin{align}\label{upperx}
    |X|\leq |A|+|B|-\frac{8}{3}t+\frac53.
\end{align}

Also, combining \eqref{holes} and \eqref{tbound}, we have
\begin{align}\label{lowerx}
    |X|\geq \frac{|A||B|-t|A|-t|B|+\frac{4}{3}t^2-\frac{2}{3}t+\frac{1}{3}+y}{|B|-t}.
\end{align}

\begin{claim}\label{claimC} $|B|\ge 2t-1$.
\end{claim}

\begin{proof}[Proof of \ref{claimC}]
    Assume to the contrary that $|B|\le 2t-2$. Let
    \begin{align*}
        |B|=t+u.
    \end{align*}
    Then $1\le u\le t-2$. Combining \eqref{upperx} and \eqref{lowerx}, we have
    \begin{align*}
        \frac{u(|A|-t)+\frac{1}{3}t^2-\frac{2}{3}t+\frac{1}{3}+y}{u}\leq |X|
        \leq |A|-\frac{5}{3}t+u+\frac53.
    \end{align*}
    As $y\ge 0$, this implies
    \begin{align}\label{fu}
        f(u):=\frac{\frac{1}{3}t^2-\frac{2}{3}t+\frac{1}{3}}{u}-u+\frac{2}{3}t-\frac53
        \leq0,
    \end{align}
    where $f(u)$ is defined on $[1,\, t-2]$. As $\frac{1}{3}t^2-\frac{2}{3}t+\frac{1}{3}\ge 0$, we know that $f(u)$ attains its minimum when $u=t-2$. However, as $t\ge 3$, we have
    \begin{align*}
        f(u)\ge f(t-2)=\frac{\frac{1}{3}t^2-\frac{2}{3}t+\frac{1}{3}}{t-2}
        -t+2+\frac{2}{3}t-\frac53=\frac{t-1}{3(t-2)}>0,
    \end{align*}
    contradicting \eqref{fu}. So we can conclude that $|B|\ge 2t-1$.
\end{proof}

If $A+_t B=A+B$, then we can simply take $A'=A$ and $B'=B$, yielding \eqref{conclusion}, contrary to assumption. Thus we may assume that
\begin{align}\label{notequal}
    A+_t B\subsetneq A+B.
\end{align}
Thus $\mathsf r_{A,B}(g_0)\leq t-1$ for some $g_0\in A+B$, whence Proposition \ref{multi} implies
\begin{align}\label{lowermulti}
    |A+B|\ge |A|+|B|-t+1.
\end{align}

Observe that $\mathsf r_{A,B}(x)=|A\cap (x-B)|=|(-x+A)\cap -B|=|(-y+A)\cap -B|=|A\cap (y-B)|=\mathsf r_{A,B}(y)$ for any $y\in x+ \mathsf H(A)$. In particular, any $\mathsf H(A)$-coset is either fully disjoint from $A+_{t+1}B$ or fully contained in $A+_{t+1}B$.

Suppose $|\mathsf{H}(A)|\ge t$. If there exists $a\in A$ and $b\in B$ with $\mathsf r_{A,\,B}(a+b)\le t$, then
\begin{align*}
    |(A+b)\setminus (A+_{t+1} B)|\ge |(a+\mathsf{H}(A))+b|\ge t,
\end{align*}
contradicting \ref{claimB}. On the other hand, if $\mathsf r_{A,\,B}(a+b)\ge t+1$ for all $a\in A$ and $b\in B$, then
\begin{align*}
    A+B=A+_{t+1} B=A+_t B,
\end{align*}
which contradicts \eqref{notequal}.
So we instead conclude that
\begin{align}\label{HA}
    |\mathsf{H}(A)|\le t-1.
\end{align}

For  any subgroup $G'\leq G$, let $\phi_{G'}: G\longrightarrow G/G'$ denote the natural homomorphism. We use the symbol $\sqcup$ to indicate a union that is  disjoint.

\begin{claim}\label{claimE}
If    $|\phi_K(B)|=1$ for some subgroup  $K\le G$, then  $|\phi_K(A)|=1$.
\end{claim}

\begin{proof}[Proof of \ref{claimE}]
This claim is trivial if $K=G$, so it suffices to prove it for $K<G$. Assume to the contrary that, for a proper subgroup $K<G$, we have $|\phi_K (A)|=r\ge 2$ and $|\phi_K (B)|=1$. Let $A=\bigsqcup_{i=1}^r A_i$ be the $K$-coset decomposition of $A$, where each $A_i$ is a $K$-coset slice, that is, a nonempty intersection of $A$ with a $K$-coset.

If there exists $i\in [1,\, r]$ such that $|A_{i}|\le t-1$, then for an arbitrary $a_i\in A_{i}$ and every $b\in B$, we have $\mathsf r_{A,\,B}(a_i+b)=\mathsf r_{A_{i},\,B}(a_i+b)\le t-1$. Thus $|(a_i+B)\setminus (A+_{t+1} B)|\ge |B|\geq t+1$, which contradicts \ref{claimB}.
Therefore we instead conclude that  $|A_{i}|\ge t$ for every $i\in [1,\, r]$.

We can now apply the induction hypothesis to each pair of $A_{i}$ and $B$, as $|A_{i}+_j B|\leq |A+_j B|$ for any $j\ge 1$ with strict inequality for $j=1$ since $r\geq 2$. If there exists $i_0\in [1,\, r]$ such that  $\sum_{j=1}^t |A_{i_0}+_j B|\ge t|A_{i_0}|+t|B|+\lceil-\frac{4}{3}t^2+\frac{2}{3}t\rceil$, then combining this  with the fact that $\sum_{j=1}^t |A+_j B|=\sum_{i=1}^r \Bigl(\sum_{j=1}^t |A_{i}+_j B|\Bigr)$ and using the trivial estimate  $\sum_{j=1}^t |A_{i}+_j B|\ge t|A_{i}|$ whenever $ i\neq i_0$, it follows that
$\sum_{j=1}^t |A+_j B|\ge t|A|+t|B|+\lceil-\frac{4}{3}t^2+\frac{2}{3}t\rceil$, contradicting \eqref{hypothesis}.
Therefore we can assume that, for any $i\in [1,\, r]$, \eqref{conclusion} and \eqref{conclusion-impliedbound} hold for  $A_{i}$ and $B$, say with $A_{i}'\subseteq A_{i}$, \ $B_i'\subseteq B$, \ $\ell_i$, \ $H_i$ and $\rho_i$ being the parameters which correspond to $A'$,  \ $B'$, \ $\ell$,\  $H$, and $\rho$ in Theorem \ref{new}. By Items 1 and 5 in Proposition \ref{mainprop}, we can w.l.o.g.~assume that $A_{i}'=(A_{i}'+H_i)\cap A_{i}$ and $B_i'=(B_i'+H_i)\cap B$ with  $|H_i|-\rho_i\ge t+1$, for all $i\in [1,r]$.

If  $A_{i}+B=A_{i}+_t B$ for every $i\in [1,\, r]$, then  $A+B=\Bigl(\bigsqcup_{i=1}^r A_{i}\Bigr)+B=\bigsqcup_{i=1}^r (A_{i}+B)=\bigsqcup_{i=1}^r (A_{i}+_t B)=\Bigl(\bigsqcup_{i=1}^r A_{i}\Bigr)+_t B=A+_t B$, which contradicts \eqref{notequal}. Thus, we can assume that there exists $j\in [1,\, r]$ such that $A_{j}+_t B\subsetneq A_{j}+B$, which forces $\ell_j\ge 1$ since \eqref{conclusion} ensures that $A_j+_tB=A'_j+B'$.
Since $\ell_j\geq 1$, there is either some $a\in A_j\setminus A'_j$ or some $b\in B\setminus B'$. By Items 4 and 5 in Proposition \ref{mainprop},  we either have $|(a+B)\setminus (A+_{t+1} B)|=|(a+B)\setminus (A_{j}+_{t+1} B)|\ge |(a+B)\setminus (A_{j}+_t B)|\ge |H_j|-\rho_j\ge t+1$ or  $|(A+b)\setminus (A+_{t+1} B)|\ge |(A_j+b)\setminus (A_j+_{t+1} B)|\geq t+1$, both contradicting \ref{claimB}.
\end{proof}

Now we apply the Dyson Transform.
For $z\in A-B$, let $A(z)=A\cup (z+B)$ and $B(z)=A\cap (z+B)$. Note $A-B$ is precisely the set of all elements $z$ with $B(z)\neq \emptyset$.  Now
\begin{align}\label{azbz}
    |A(z)|+|B(z)|=|A|+|B|.
\end{align}
Also, for any $g\in G$, we have
\begin{align*}
    \mathsf r_{A,\,z+B}(g)&=\mathsf r_{A,\,A\cap(z+B)}(g)+\mathsf r_{A,\,(z+B)\setminus A}(g)\\
    &=\mathsf r_{A,\,A\cap(z+B)}(g)+\mathsf r_{A\cap(z+B),\,(z+B)\setminus A}(g)+\mathsf r_{A\setminus(z+B),\,(z+B)\setminus A}(g)\\
    &=\mathsf r_{A\cup(z+B),\,A\cap(z+B)}(g)+\mathsf r_{A\setminus(z+B),\,(z+B)\setminus A}(g)\\
    &=\mathsf r_{A(z),\,B(z)}(g)+\mathsf r_{A\setminus(z+B),\,(z+B)\setminus A}(g)\\
    &\ge \mathsf r_{A(z),\,B(z)}(g),
\end{align*}
so for any $i\ge 1$, we have
\begin{align}\label{azbz2}
    A(z)+_i B(z)\subseteq A+_i (z+B)=z+(A+_i B).
\end{align}

If, for every $z\in A-B$, we have $|B(z)|=|B|$, meaning $z+B\subseteq A$, then $A-B+B\subseteq A$, so $A-B+B=A$, which implies $-B+B\subseteq \mathsf{H}(A)$, and thus $|\mathsf{H}(A)|\ge |-B+B|\ge |B|\ge t+1$, which contradicts \eqref{HA}. Therefore there exists $z\in A-B$ such that $|B(z)|\le |B|-1$. Let  $z'\in A-B$ be an element such that
\begin{align*}
    |B(z')|=\max\{|B(z)|: \; z\in A-B,\ 1\le B(z)\le |B|-1\}.
\end{align*}
Then, for our convenience, by replacing $B$ with $B+z'$, we can assume that
\begin{align*}
    |B(0)|=\max\{|B(z)|:\;  z\in A-B,\ 1\le B(z)\le |B|-1\}.
\end{align*}
We would like to mention that, in fact, $A(0)=A\cup B$ and $B(0)=A\cap B$.

\begin{claim}\label{claimD}
$|B(0)|\ge t$.
\end{claim}

\begin{proof}[Proof of \ref{claimD}]
Assume to the contrary that $|B(0)|\le t-1$. Then, by the maximality of $|B(0)|$,
\begin{align}\label{inA}
    |(g+B)\cap A|\ge t\quad\mbox{ implies }\quad g+B\subseteq A
\end{align}
for any $g\in G$.

Let
\begin{align}\label{mins}
    s:=\min\{|(b+A)\setminus X|:\; b\in B\},
\end{align}
and let $b_0\in B$ be an element such that the equality in \eqref{mins} is attained. If we translate $A$ and $B$ each by $-b_0$, then we get new sets $A_0:=-b_0+A$, $B_0:=-b_0+B$, and $X_0:=A_0+_{t+1} B_0=-2b_0+X$. As $b_0\in B$ attains the minimum in \eqref{mins}, we have
\begin{align}\label{bAX}
    |(b_0+A)\setminus X|\le |(b+A)\setminus X|
\end{align}
for any $b\in B$. Since $|(b_0+A)\setminus X|=|(-b_0+A)\setminus (-2b_0+X)|=|A_0\setminus X_0|$ and $|(b+A)\setminus X|=|(b-2b_0+A)\setminus (-2b_0+X)|=|(b-b_0+A_0)\setminus X_0|$, \eqref{bAX} implies
\begin{align*}
    |A_0\setminus X_0|\le |(b'+A_0)\setminus X_0|
\end{align*}
for any $b'\in -b_0+B=B_0$. Thus, by replacing $A$ and $B$ with $A_0$ and $B_0$, we can w.l.o.g.~assume that $b_0=0\in B$ attains the minimum in \eqref{mins}. This means
\begin{align}\label{AcapX}
    |A\cap X|=|A|-s.
\end{align}
Note that $s$ is the minimum degree of a vertex from $B$ in the graph $\Gamma_t$ defined earlier, so
\begin{align}\label{EsB}
    |E|\ge s|B|.
\end{align}

Let $T=\{x\in A:\; x+B\subseteq A\}$. Then $T+B\subseteq A$.
Let $\X:=\mathsf{H}(T+B)$. If $x\in T+\X$, then $T$ is nonempty and $x+B\subseteq T+\X+B=T+B\subseteq A$, so $x\in T$ by the definition of $T$. Thus $T$ is $\X$-periodic, whence  $\X\subseteq \mathsf{H}(T)\subseteq \mathsf{H}(T+B)=\X$, so $$\mathsf H(T+B)=\X=\mathsf{H}(T).$$

Now we prove
\begin{align}\label{boundT}
    |T|\le |A|-|B|+t-1,
\end{align}
which is trivially true if $T=\emptyset$ since $|A|\geq |B|$ and $t\geq 1$.

Assume that $T\neq \emptyset$. If $|\phi_\X (B)|=1$, then $|\X|\ge |B|\ge t+1$, and by \ref{claimE}, we have $|\phi_\X (A)|=1$. As $\X=\mathsf{H}(T+B)$ and $T+B\subseteq A$, we know that an $\X$-coset is contained in $A$, which, together with $|\phi_\X (A)|=1$, forces $A$ to equal this $\X$-coset, so $|\mathsf{H}(A)|=|\X|\ge t+1$, which contradicts \eqref{HA}. Therefore $|\phi_\X (B)|\ge 2$ when $T\neq \emptyset$.
Suppose that there exists $b^*\in B$ with $|(b^*+\X)\cap B|\ge t$. We have $B\setminus (b^*+\X)\neq \emptyset$ as $|\phi_\X (B)|\ge 2$. Then, for any $b^{**}\in B\setminus (b^*+\X)$ and any $z\in T$, we have
\begin{align*}
    b^{**}-b^*+z+(b^*+\X)\cap B\subseteq b^{**}+z+\X\subseteq B+T+\X=T+B\subseteq A,
\end{align*}
which implies
\begin{align*}
    \bigl|\bigl(b^{**}-b^*+z+(b^*+\X)\cap B\bigr)\cap A\bigr|=|(b^*+\X)\cap B|\ge t,
\end{align*}
and thus
\begin{align*}
    |(b^{**}-b^*+z+B)\cap A|\ge t.
\end{align*}
So, by \eqref{inA}, we have $b^{**}-b^*+z+B\subseteq A$, and thus $b^{**}-b^*+z\in T$ for all $z\in T$, which further implies that $b^{**}-b^*\in \mathsf{H}(T)=\X$. But this contradicts that $b^{**}\in B\setminus (b^*+\X)$. So we instead concldude that, for any $b\in B$, we have
\begin{align}\label{capt-1}
    |(b+\X)\cap B|\le t-1\quad\mbox{ when $T\neq \emptyset$}.
\end{align}

If \eqref{boundT} fails, then $|T|\ge |A|-|B|+t$, which, together with the fact that $T+B\subseteq A$, implies $|T+B|\le |A|\le |T|+|B|-t$. So, by Kneser's Theorem (Theorem \ref{kneserthm}), we have $|T+B|=|T|+|B|-|\X|+\rho$, where $\rho:=|(T+\X)\setminus T|+|(B+\X)\setminus B|$. Hence
\begin{align*}
    |T|+|B|-t\ge |T+B|=|T|+|B|-|\X|+\rho,
\end{align*}
which means $|\X|-\rho\ge t$, and thus $|(b+\X)\cap B|\ge |\X|-\rho\ge t$ for $b\in B$, which contradicts \eqref{capt-1}. With this, \eqref{boundT} is now established.

By \ref{claimC}, we have $|B|\ge 2t-1$, so \eqref{boundT} implies that
\begin{align*}
    |A\setminus T|=|A|-|T|\ge |B|-t+1\ge t.
\end{align*}
Let $\{x_1,\ x_2,\ ...,\ x_t\}\subseteq A\setminus T$ be a subset of size $t$. By definition of $T$, we have $x_i+B\nsubseteq A$ for each $i\in [1,t]$. Thus \eqref{inA} ensures
\begin{align}\label{xiB}
    |(x_i+B)\setminus A|=|x_i+B|-|(x_i+B)\cap A|\ge |B|-t+1
\end{align}
for any $i\in [1,\, t]$.

Then, as explained below, we have
\begin{align}\label{AtBlower}
    \sum_{i=1}^t |A+_i B|\ge t|A\cap X|+\sum_{j=1}^t |(x_j+B)\setminus A|+|E|-\sum_{k=1}^t |(x_k+B)\setminus (A\cup X)|.
\end{align}
Each element from $A\cap X\subseteq X=A+_{t+1} B$ contributes $t$ to $\sum_{i=1}^t |A+_i B|$. The contributions from each $(x_j+B)\setminus A$ are disjoint from the contributions from $A\cap X$ as they do not live in $A$ while being cumulative as any element from $\Bigl(\bigcup_{j=1}^t (x_j+B)\Bigr)\setminus A$ is counted at most $t$ times (at most once for each $x_j\in A$). Each element from $E$ contributes $1$ to $\sum_{i=1}^t |A+_i B|$, and they are disjoint from all prior contributions except for elements from $(x_k+B)\setminus (A\cup X)$. Thus we need to subtract these elements to avoid double counting.

We have the following estimates.
\begin{itemize}
    \item By \eqref{AcapX}, we have $|A\cap X|=|A|-s$.
    \item By \eqref{xiB}, we have $|(x_j+B)\setminus A|\ge |B|-t+1$ for any $j\in [1,\, t]$.
    \item By \eqref{EsB}, we have $|E|\ge s|B|$.
    \item By \eqref{claima111222}, we have $\sum_{k=1}^t |(x_k+B)\setminus (A\cup X)|\le \sum_{k=1}^t |(x_k+B)\setminus X|\le t(t-1)$.
\end{itemize}
So, in \eqref{AtBlower}, we have
\begin{align}
    \sum_{i=1}^t |A+_i B|&\ge t(|A|-s)+t(|B|-t+1)+s|B|-t(t-1)\nonumber \\
    &=t|A|+t|B|-2t(t-1)+(|B|-t)s. \label{green}
\end{align}

For each $i\in [1,\, t]$, we have $|(x_i+B)\cap (X\setminus A)|+|(x_i+B)\setminus (X\cup A)|=|(x_i+B)\setminus A|\ge |B|-t+1$ by \eqref{xiB}. We have $|(x_i+B)\setminus (X\cup A)|\le |(x_i+B)\setminus X|\le t-1$ by \eqref{claima111222}. Combining both these estimates yields $|X\setminus A|\ge |(x_i+B)\cap (X\setminus A)|\ge |B|-2t+2$, and thus (in view also of \eqref{AcapX})
\begin{align*}
    |X|=|X\cap A|+|X\setminus A|\ge |A|-s+|B|-2t+2,
\end{align*}
which, together with \eqref{upperx}, implies that
\begin{align*}
    s>\frac{2}{3}t.
\end{align*}

Applying the estimate $s>\frac{2}{3}t$ in \eqref{green}, we obtain (since $|B|\geq t+1$)
\begin{align*}
    \sum_{i=1}^t |A+_i B|>t|A|+t|B|-2t(t-1)+(|B|-t)\cdot\frac{2}{3}t.
\end{align*}
Then, as $|B|\ge 2t-1$ by \ref{claimC}, we have
\begin{align*}
    \sum_{i=1}^t |A+_i B|&>t|A|+t|B|-2t(t-1)+(t-1)\cdot\frac{2}{3}t\\
    &=t|A|+t|B|-\frac{4}{3}t^2+\frac{4}{3}t,
\end{align*}
which contradicts \eqref{tbound}. The proof of \ref{claimD} is completed, and we conclude that $|B(0)|\ge t$.
\end{proof}

If $\sum_{i=1}^t|A(0)+_iB(0)|\geq t|A(0)|+t|B(0)|+\lceil-\frac{4}{3}t^2+\frac{2}{3}t\rceil$, then combining this with \eqref{azbz} and \eqref{azbz2} contradicts \eqref{tbound}. Therefore $\sum_{i=1}^t|A(0)+_iB(0)|< t|A(0)|+t|B(0)|+\lceil-\frac{4}{3}t^2+\frac{2}{3}t\rceil$
Now, by \eqref{azbz}, \eqref{azbz2}, $|B(0)|<|B|$ and \ref{claimD}, we can apply the induction hypothesis to $A(0)$ and $B(0)$.
Thus
\be\label{goblinjuice}\ell:=|A(0)\setminus A'|+|B(0)\setminus B'|\leq t-1\quad\und\quad A'+_tB'=A'+B'=A(0)+_tB(0)\ee for some $A'\subseteq A(0)$ and $B'\subseteq B(0)$.
Let $H=\mathsf H(A'+B')$ and $\rho=|(H+A')\setminus A'|+|(H+B')\setminus B'|$ be as defined in Theorem \ref{new}.
By Item 1 in Proposition \ref{mainprop}, we can assume that \be\label{oncemore}A'=(A'+H)\cap A(0)\quad\und\quad B'=(B'+H)\cap B(0).\ee
 By Item 5 in Proposition \ref{mainprop}, we have \be\label{Hbig}|H|-\rho\ge t+1.\ee

By \eqref{azbz2}, for any $i\ge 1$, we have
\begin{align}\label{mid}
    A'+_i B'\subseteq A(0)+_i B(0)\subseteq 0+(A+_i B)=A+_i B.
\end{align}

In particular, we have
\begin{align}\label{a'b'atb}
    A'+B'=A'+_t B'=A(0)+_t B(0)\subseteq A+_t B.
\end{align}

By \eqref{mid}, we have
\begin{align*}
    A'+_i B'\subseteq (A+_i B)\cap (A'+B')\quad\mbox{ for any $i\geq 1$},
\end{align*}
which, together with Item 3 in Proposition \ref{mainprop} and \eqref{azbz}, implies
\begin{align}
    \sum_{i=1}^t |(A+_i B)\cap (A'+B')|&\ge \sum_{i=1}^t |A'+_i B'| \label{imp1} \\
    &=t|A|+t|B|-t\ell-t(|H|-\rho). \label{important}
\end{align}
Depending on whether $A'=A(0)$ and whether $B'=B(0)$, there are three cases.

\smallskip

\textbf{Case 1.} There exists $a\in A(0)\setminus A'$.

In this case, we have $\ell=|A(0)\setminus A'|+|B(0)\setminus B'|\ge 1$ (by \eqref{goblinjuice}). By \eqref{oncemore} and Item 4 in Proposition \ref{mainprop}, there exists $\beta\in B'$ such that
\begin{align}\label{case1star}
    \Bigl(a+\bigl((\beta+H)\cap B(0)\bigr)\Bigr)\cap (A(0)+_t B(0))=\Bigl(a+\bigl((\beta+H)\cap B'\bigr)\Bigr)\cap (A(0)+_t B(0))=\emptyset.
\end{align}
We have $a\in A(0)=A\cup B$ and $(\beta+H)\cap B(0)\subseteq B(0)=A\cap B$. Thus, $a+\bigl((\beta+H)\cap B(0)\bigr)\subseteq A+B$.

Recall that $a\in A(0)=A\cup B$ and $B(0)=A\cap B$. Suppose that, in $a+\bigl((\beta+H)\cap B(0)\bigr)$, there are at least $t$ elements each with at most $t$ representations in $A+B$.
Then $|(a+B(0))\setminus (A+_{t+1}B)|\geq t$.
 If $a\in A$, then $|(a+B)\setminus (A+_{t+1}B)|\geq |(a+B(0))\setminus (A+_{t+1}B)|\geq t$, and if $a\in B$, then $|(B+a)\setminus (A+_{t+1}B)|\geq |(a+B(0))\setminus (A+_{t+1}B)|\geq t$, both contradicting \ref{claimB}.
  So we instead conclude that there are at most $t-1$ elements in $a+\bigl((\beta+H)\cap B(0)\bigr)$ each  having  at most $t$ representations in $A+B$. Thus, as $|\bigl((\beta+H)\cap B(0)\bigr)|\geq |H|-p\geq t+1$ by definition of $\rho$ and \eqref{Hbig}, it follows that there is a subset $S\subseteq a+\bigl((\beta+H)\cap B(0)\bigr)$ with \be\label{Scard}|S|=|H|-\rho-(t-1)\quad\und\quad  \mathsf r_{A,\,B}(s)\ge t+1\quad\mbox{ for every $s\in S$}.\ee

Now we estimate $\sum_{i=1}^t |(A+_i B)\setminus (A'+B')|$. By \eqref{case1star} and \eqref{a'b'atb}, $S\subseteq a+\bigl((\beta+H)\cap B(0)\bigr)$ is disjoint from $A(0)+_t B(0)=A'+B'$ (with the equality by \eqref{goblinjuice}), while $S\subseteq A+_t B$ by definition of $S$. Thus each element in $S$ contributes $t$ to $\sum_{i=1}^t |(A+_i B)\setminus (A'+B')|$. By Item 4  in Proposition \ref{mainprop}, we know that, for any $e\in A(0)\setminus A'$, there are at least $|H|-\rho$ elements in $e+B'$ that are not in $A(0)+_t B(0)=A'+B'$ (with the equality  from \eqref{goblinjuice}), and at most $|S|$ of them are equal to an element in $S$. Similarly, for any $e'\in B(0)\setminus B'$, there are at least $|H|-\rho$ elements in $A'+e'$ that are not in $A(0)+_t B(0)=A'+B'$, and at most $|S|$ of them are equal to an element in $S$. In total, there are $\ell=|A(0)\setminus A'|+|B(0)\setminus B'|\ge 1$ such elements $e$ and $e'$ (noted at the start of Case 1). Each one  contributes at least $|H|-\rho-|S|$ to $\sum_{i=1}^t |(A+_i B)\setminus (A'+B')|$ with the contributions cumulative as  $\ell\leq t-1$ (by \eqref{goblinjuice}). Thus,
\begin{align*}
    \sum_{i=1}^t |(A+_i B)\setminus (A'+B')|\ge t|S|+\ell(|H|-\rho-|S|).
\end{align*}

Then, by \eqref{important}, \eqref{Scard} and $\ell\le t-1$, we have
\begin{align*}
    \sum_{i=1}^t |A+_i B|&=\sum_{i=1}^t |(A+_i B)\cap (A'+B')|+\sum_{i=1}^t |(A+_i B)\setminus (A'+B')|\\
    &\ge t|A|+t|B|-t\ell-t|H|+t\rho+t|S|+\ell(|H|-\rho-|S|)\\
    &=t|A|+t|B|-t^2+t-\ell\\
    &\ge t|A|+t|B|-t^2+1,
\end{align*}
which contradicts \eqref{tboundaux}. So Case 1 is eliminated.

\smallskip

\textbf{Case 2.} $A'=A(0)$, but there exists $b\in B(0)\setminus B'$.

In this case, we also have $\ell=|A(0)\setminus A'|+|B(0)\setminus B'|\ge 1$ (by \eqref{goblinjuice}). By \eqref{oncemore} and Item 4 in Proposition \ref{mainprop}, there exists $\alpha\in A'=A(0)$ such that
\begin{align}\label{case2star}
    \Bigl(\bigl((\alpha+H)\cap A(0)\bigr)+b\Bigr)\cap (A(0)+_t B(0))=\Bigl(\bigl((\alpha+H)\cap A'\bigr)+b\Bigr)\cap (A(0)+_t B(0))=\emptyset.
\end{align}
Recall that $A(0)=A\cup B$. Let $A_\alpha=(\alpha+H)\cap A$ and $B_\alpha=(\alpha+H)\cap B$. Then
\begin{align*}
    |A_\alpha|+|B_\alpha|-|A_\alpha\cap B_\alpha|=|(\alpha+H)\cap (A\cup B)|=|(\alpha+H)\cap A(0)|\ge |H|-\rho,
\end{align*}
with the inequality by definition of $\rho$, which means
\begin{align*}
    |A_\alpha|+|B_\alpha|\ge |H|-\rho+|A_\alpha\cap B_\alpha|,
\end{align*}
and thus we can find $A_\alpha'\subseteq A_\alpha$ and $B_\alpha'\subseteq B_\alpha$ with $A_\alpha'\cap B_\alpha'=\emptyset$ and $|A_\alpha'|+|B_\alpha'|=|H|-\rho\geq t+1$ (by Item 4 in Proposition \ref{mainprop}).

If, in $(A_\alpha'+b)\sqcup (b+B_\alpha')$, there are at least $2t-1$ elements each with at most $t$ representations in $A+B$, then the Pigeonhole Principle ensures that
\begin{align*}
    |(A_\alpha'+b)\setminus (A+_{t+1} B)|\ge t \quad \text{or} \quad |(b+B_\alpha')\setminus (A+_{t+1} B)|\ge t.
\end{align*}
Since $b\in B(0)=A\cap B$, \ $A'_\alpha\subseteq A_\alpha\subseteq A$ and $B'_\alpha\subseteq B_\alpha\subseteq B$, both cases above contradict \ref{claimB}.
Therefore, we instead conclude that, in $(A_\alpha'+b)\sqcup (b+B_\alpha')$, there are at most $2t-2$ elements each with at most $t$ representations in $A+B$. We observe that
\begin{itemize}
    \item either $|H|-\rho\le 2t-2$
    \item or $|H|-\rho\ge 2t-1$ and there is a subset $S\subseteq (A_\alpha'+b)\sqcup (b+B_\alpha')$ with
    \begin{align}\label{sizeS}
        |S|=|(A_\alpha'+b)\sqcup (b+B_\alpha')|-(2t-2)=|H|-\rho-(2t-2)
    \end{align}
    and $\mathsf r_{A,\,B}(s)\ge t+1$ for every $s\in S$.
\end{itemize}

Note that, in the former case,  we trivially have
\begin{align*}
    |(A+_t B)\setminus(A'+_t B')|\ge |H|-\rho-(2t-2).
\end{align*}
In the latter case, note that
$A'_\alpha\subseteq A_\alpha=(\alpha+H)\cap A\subseteq (\alpha+H)\cap A(0)$ and $B'_\alpha\subseteq B_\alpha=(\alpha+H)\cap B\subseteq (\alpha+H)\cap A(0)$. Thus \eqref{a'b'atb} and \eqref{case2star} ensure
\begin{align}\label{S-starter}
    S\cap (A'+_t B')=S\cap (A(0)+_t B(0))\subseteq \bigl((A_\alpha'+b)\sqcup (b+B_\alpha')\bigr)\cap (A(0)+_t B(0))=\emptyset,
\end{align} and so as $S\subseteq A+_{t+1} B\subseteq A+_t B$ (by definition of $S$),  we also have
\begin{align}\label{anothercite}
    |(A+_t B)\setminus(A'+_t B')|\ge |S|=|H|-\rho-(2t-2).
\end{align}

\smallskip

\textbf{Subcase 2.1.} $\ell\le \frac{2}{3}t$.

By \ref{claimA}, we have
\begin{align*}
    \sum_{i=1}^t|A+_i B|=&\sum_{i=1}^{t-1}|A+_i B|+|A+_t B| \\
    \ge& (t-1)|A|+(t-1)|B|-\frac{4}{3}(t-1)^2+\frac{2}{3}(t-1) \\
    &+|A'+_t B'|+|(A+_t B)\setminus(A'+_t B')|,
\end{align*}
where we have $|A+_t B|=|A'+_t B'|+|(A+_t B)\setminus(A'+_t B')|$ because $A'+_t B'=A(0)+_t B(0)\subseteq A+_t B$ by \eqref{a'b'atb}.
By \eqref{a'b'atb} and Kneser's Theorem (Theorem \ref{kneserthm}), we have $|A'+_t B'|=|A'+B'|\ge |A'+H|+|B'+H|-|H|$, which combined with   \eqref{azbz} and \eqref{anothercite} yields \begin{align*}
    |A'+_t B'|+|(A+_t B)\setminus(A'+_t B')|&\ge |A'+H|+|B'+H|-|H|+|H|-\rho-(2t-2)\\
    & =|A'|+|B'|-(2t-2)\\& =|A(0)|+|B(0)|-\ell-(2t-2)\\&=|A|+|B|-\ell-(2t-2).
\end{align*}
Combining the previous two estimates along with the subcase hypothesis $\ell\le \frac{2}{3}t$, it follows that  $\sum_{i=1}^t|A+_i B|\geq t|A|+t|B|-\frac{4}{3}t^2+\frac{2}{3}t$,
which contradicts \eqref{tbound}.

\smallskip

\textbf{Subcase 2.2.} $\ell>\frac{2}{3}t$ and $|H|-\rho\le 2t-2$.

Recall that $A'+B'=A(0)+_tB(0)$ by \eqref{goblinjuice}. We have  $A'=A(0)$ by hypothesis of Case 2, so
$\ell =|A(0)\setminus A'|+|B(0)\setminus B'|=|B(0)\setminus B'|>\frac{2}{3}t$ by \eqref{goblinjuice} and hypothesis of Subcase 2.2.
By Item 4 in Proposition \ref{mainprop}, we know that, for each of the $\ell$ elements  $e'\in B(0)\setminus B'$, there are at least $|H|-\rho$  elements in $(A'+e')\setminus (A'+B')$. Hence, since $A'+e\subseteq A(0)+B(0)\subseteq A+B$  (by \eqref{mid}) and $\ell\leq t-1$ (by \eqref{goblinjuice}), it follows that $\sum_{i=1}^t |(A+_i B)\setminus (A'+B')|\geq \ell(|H|-\rho).$ Combining this inequality along with
 \eqref{important} yields
\begin{align*}
    \sum_{i=1}^t |A+_i B|&=\sum_{i=1}^t |(A+_i B)\cap (A'+B')|+\sum_{i=1}^t |(A+_i B)\setminus (A'+B')|\\
    &\ge t|A|+t|B|-t\ell+(\ell-t)(|H|-\rho)\\
    &\ge t|A|+t|B|-t\ell+(\ell-t)(2t-2)\\
    &=t|A|+t|B|-2t^2+2t+\ell(t-2) \\
    &\geq t|A|+t|B|-\frac{4}{3}t^2+\frac{2}{3}t,
\end{align*}
where the second inequality follows because $\ell-t<0$ (by \eqref{goblinjuice}) and $|H|-\rho\le 2t-2$ (by hypothesis of Subcase 2.2), and the third inequality  because $\ell>\frac{2}{3}t$ (by hypothesis of Subcase 2.2) and $t-2\geq 0$. But now we have a contradiction to \eqref{tbound}.

\smallskip

\textbf{Subcase 2.3.} $\ell>\frac{2}{3}t$ and $|H|-\rho\ge 2t-1$.

In this subcase, as  explained in \eqref{sizeS}, there is a subset $S\subseteq (A_\alpha'+b)\sqcup (b+B_\alpha')$ with $$|S|=|H|-\rho-(2t-2)\;\und\; \mathsf r_{A,\,B}(s)\ge t+1$$ for every $s\in S$. By \eqref{goblinjuice} and \eqref{S-starter}, we have $S\cap (A'+B')=S\cap (A(0)+_t B(0))=\emptyset$. So each element in $S$ contributes $t$ to the sum $\sum_{i=1}^t |(A+_i B)\setminus (A'+B')|$.
We have  $A'=A(0)$ by hypothesis of Case 2, so we have
$\ell =|A(0)\setminus A'|+|B(0)\setminus B'|=|B(0)\setminus B'|>\frac{2}{3}t$ by \eqref{goblinjuice} and hypothesis of Subcase 2.3.
By Item 4 in Proposition \ref{mainprop}, we know that, for each of the $\ell$ elements  $e'\in B(0)\setminus B'$, there are at least $|H|-\rho$  elements in $(A'+e')\setminus (A'+B')$, and at most $|S|$ of them are equal to an element in $S$. Hence, since $A'+e\subseteq A(0)+B(0)\subseteq A+B$  (by \eqref{mid}) and $\ell\leq t-1$ (by \eqref{goblinjuice}), it follows that $\sum_{i=1}^t |(A+_i B)\setminus (A'+B')|\geq t|S|+\ell(|H|-\rho-|S|)=t(|H|-\rho)-(t-\ell)(2t-2).$ Combining this inequality along with
 \eqref{important} yields
\begin{align*}
    \sum_{i=1}^t |A+_i B|&=\sum_{i=1}^t |(A+_i B)\cap (A'+B')|+\sum_{i=1}^t |(A+_i B)\setminus (A'+B')|\\
    &\ge t|A|+t|B|-2t^2+2t+\ell(t-2) \\
    &\geq t|A|+t|B|-\frac{4}{3}t^2+\frac{2}{3}t,
\end{align*}
where the second inequality follows because $\ell>\frac{2}{3}t$ (by hypothesis of Subcase 2.3) and $t-2\geq 0$. But now we have a contradiction to \eqref{tbound}, completing Case 2.

\smallskip

\textbf{Case 3.} $A'=A(0)$ and $B'=B(0)$.

The analysis of Case 3 is the last part of this proof, and we omit it here because it is a verbatim copy of Case 2 in the proof of Theorem \ref{old} presented in \cite[p.~168]{Gry2}, which spans 5--6 pages. Note that, in our notation, $A(0)$ and $B(0)$ play the same roles as $A(z)$ and $B(z)$ in \cite{Gry2}; and \eqref{tboundaux}, \eqref{notequal}, \eqref{Hbig}, \eqref{a'b'atb}, \eqref{imp1}, and \eqref{important} in this proof are the same as (12.16), (12.23), (12.41), (12.42), (12.43), and (12.44) in \cite{Gry2}, respectively, with the argument of Case 2 in \cite{Gry2} otherwise self-contained. The inductive parameter $-(|A|+|B|)$, not used in prior arguments, is used for this case.
With this,  the proof of Theorem \ref{new} is completed.
\end{proof}


\subsection{Further Possibilities}\label{sec-conj}

As noted in the introduction, the goal is to achieve a common generalization of Pollard and Kneser's Theorems with Theorems \ref{new} progress in this direction. The naive hope would be that Theorem \ref{new} remains true replacing the hypothesis \eqref{hypothesis} with the target hypothesis $\sum_{i=1}^t|A+_iB|<t|A|+t|B|-t^2$. However, an example given in \cite{Gry1} shows that \eqref{hypothesis} can be replaced, at best, by the bound
    \begin{align}\label{98-bound}
        \sum_{i=1}^t |A+_i B|< t|A|+t|B|-\frac98 t^2,
    \end{align}
meaning the conclusions of Theorem \ref{new} are too strong when $\sum_{i=1}^t |A+_i B|\geq  t|A|+t|B|-\frac98 t^2$.
Incidently, this means that, when $t=3$, the best possible hypothesis for $\sum_{i=1}^3|A+_iB|$ that could guarantee all conclusions in Theorem \ref{new} are true is $\sum_{i=1}^3|A+_iB|<3|A|+3|B|-10$, and
when $t=4$, the best possible hypothesis for $\sum_{i=1}^4|A+_iB|$ that could guarantee all conclusions in Theorem \ref{new} are true is $\sum_{i=1}^4|A+_iB|<4|A|+4|B|-18$, both which match the hypothesis \eqref{hypothesis} in our Theorem \ref{new}, meaning \eqref{hypothesis} is optimal for $t\in [2,4]$.

Determining what a possible Kneser-Pollard Theorem might look like is surprisingly challenging, and is likely a reason why progress in this direction is slow. 
Formulating even a potential precise statement is quite some work. In order to help spur progress, we give one of the strongest possibilities below as a starting point, which we state  as a conjecture for lack of a better word, though it may be necessary to leave open the possibly of altering portions should further investigation prove this necessary.

    \begin{conjecture}\label{conjecture}
        Let $t$ be a positive integer, let $G$ be an abelian group, and let $A,\, B\subseteq G$ be finite subsets with $|A|,\, |B|\ge t$. If
    \begin{align*}
        \sum_{i=1}^t |A+_i B|< t|A|+t|B|-t^2,
    \end{align*}
    then there exist subsets $A'\subseteq A$ and $B'\subseteq B$ such that
    \begin{align}\nn
    &|A\setminus A'|+|B\setminus B'|<u,\quad  A'+B'=A'+_uB'=A+_uB, \\
    \nn
    &H:=\mathsf H(A'+B')=\mathsf H(A'+_uB')=\mathsf H(A+_t B),\quad\und\\
    &\label{conjbound}\sum_{i=1}^t|A+_iB|\geq t|A|+t|B|-t^2-u(|H|-u)
    \end{align}
    where $t=s|H|+u$ with $u\in [1,|H|]$ and $s\geq 0$ integers.
    \end{conjecture}

In the case $t=1$, the lower bound \eqref{conjbound} becomes $|A+B|\geq |A|+|B|-|H|$ with $H=\mathsf H(A'+B')=\mathsf H(A+B)$, which is an equivalent formulation of Kneser's Theorem. The bound \eqref{conjbound}  is minimized when $u=\frac12|H|$, in which case
\be\nn
\sum_{i=1}^t|A+_iB|\geq t|A|+t|B|-t^2-\frac14|H|^2,
\ee
matching the similar bound from Theorem \ref{ham-serr-thm}. When $|H|\geq t$, we have $u=t$, and the conclusions of Theorem \ref{new} are recovered. When this is not the case, that is, when  $|H|<t$, then  $t=s|H|+u$ with $s\geq 1$, in which case \eqref{conjbound}  implies
\begin{align*}
\sum_{i=1}^t|A+_iB|&\geq t|A|+t|B|-t^2+\frac{(s+1)u^2-tu}{s}\\
&\geq t|A|+t|B|-t^2-\frac{t^2}{4s(s+1)}\geq t|A|+t|B|-\frac98t^2,\end{align*} which would force  Theorem \ref{new}'s conclusion  to hold when $\sum_{i=1}^t|A+_iB|< t|A|+t|B|-\frac98t^2$, matching what is forced by the example from \cite{Gry1} mentioned above. There could also be a strengthening of \eqref{conjbound} involving  $\rho=|(A'+H)\setminus A'|+|(B'+H)\setminus B'|< |H|-u$ and $\ell:=|A\setminus A'|+|B\setminus B'|<u$, such as
\be\label{strongconjbound}
\sum_{i=1}^t|A+_iB|\geq t|A|+t|B|-t^2-(u-\ell)(|H|-\rho-u).
\ee

We suspect that more can be added to Conjecture \ref{conjecture} regarding the structure of $A'$ and $B'$, though we warn the reader that any such structure may be surprisingly complicated and potentially similar to the Kemperman Structure Theorem, where a small number of basic building block examples can be iteratively combined through recursive processes to yield the general form for $A'$ and $B'$, with these complications only occurring when either  $s\geq 1$ or maybe $s\geq 2$, thus lying above the $-\frac98t^2$  threshold where we currently have no results and explaining why they do not occur in Theorem \ref{new}.

For instance, if $A=-B$ with $H= \mathsf H(A-A)=\mathsf H(A)$ and $|A|=(s+1)|H|=t+|H|-u$, where $t=s|H|+u$ with $s\geq 1$ and $u\in [1,|H|-1]$, then $\sum_{i=1}^t|A+_iB|=|A|^2-(|H|-u)|H|=t|A|+t|B|-u(|H|-u)$ yet $A$ can be nearly arbitrary. This example is  derived as a variant from the characterization of equality in Pollard's Theorem given in \cite{pollard-equality} and shows  that strong structure for very large $t$ is not possible. For another example, we can take any $A$ and $B$ with $|A+B|<|A|+|B|-1$ and $|H|>t$, where $H=\mathsf H(A+B)$, and find $\sum_{i=1}^t|A+_iB|=t|A|+t|B|-t|H|<t|A|+t|B|-t^2$. For a third example, take $A$ and $B$ to be arithmetic progressions of $H$-cosets (of common difference) with $t=s|H|+u$, where $s\geq 0$ and $u\in [1,|H|-1]$. Then 
\begin{align*}
\sum_{i=1}^t|A+_iB|&=t|A|+t|B|-(2s+1)|H|+2(|H|+2|H|+\ldots+s|H|)|H|\\
&=t|A|+t|B|-(2s+1)t|H|+s(s+1)|H|^2\\
&=t|A|+t|B|-t^2-u(|H|-u). 
\end{align*}
In all three cases, it is also possible to remove a small number of elements from $A$ and $B$ and stay  below the threshold $t|A|+t|B|-t^2$. These should then all be examples of basic building blocks, though it is possible others may exist.

As for examples of recursive processes, start with the first example above, say $A=-B$ and $K=\mathsf H(A-A)=\mathsf H(A)$ with $t=s|K|+u$, \ $u\in [1,|K|-1]$ and $|A|=|B|=(s+1)|K|$. By translation, we can assume $K\subseteq A\cap B$. Now define $A^*=(A\setminus K)\cup A_0$ and $B^*=(B\setminus K)\cup B_0$, where $A_0,B_0\subseteq K$ are any pair of subsets with $|A_0|,\,|B_0|\geq u$ and $\sum_{i=1}^{u}|A+_iB|<u|A_0|+u|B_0|-u^2$. Then  $H:=\mathsf H(A+_tB)=\mathsf H(A_0+_uB_0)\leq K$ and \begin{align*}
\sum_{i=1}^t|A^*+_iB^*|&=|A^*||B^*|-|A_0||B_0|+\sum_{i=1}^u|A_0+_iB_0|\\
&=t|A^*|+t|B^*|-t^2-u|A_0|-u|B_0|+u^2+\sum_{i=1}^u|A_0+_iB_0|\\
&<t|A^*|+t|B^*|-t^2.
\end{align*}
Alternatively, we could also start with any pair of subsets $A$ and $B$ with $K=\mathsf H(A)=\mathsf H(B)=\mathsf H(A+B)$ and $|A+B|<|A|+|B|-1$ such that $|K|>t$ and $A+B$ has a unique expression element modulo $K$. By translation, we can assume this unique expression (modulo $K$) element is $0$ with $K\subseteq A\cap B$. Define $A^*=(A\setminus K)\cup A_0$ and $B^*=(B\setminus K)\cup B_0$, where $A_0,B_0\subseteq K$ are any pair of subsets with $|A_0|,\,|B_0|\geq t$ and $\sum_{i=1}^{t}|A+_iB|<t|A_0|+t|B_0|-t^2$.
Then  $H:=\mathsf H(A+_tB)=\mathsf H(A_0+_tB_0)\leq K$ and \begin{align*}\sum_{i=1}^t|A^*+_iB^*|&=t(|A^*|-|A_0|)+t(|B^*|-|B_0|)
+\sum_{i=1}^t|A_0+_iB_0|<t|A^*|+t|B^*|-t^2
\end{align*}
Alternating between both recursive processes defined above and ending with a basic building block pair  $A_0$ and $B_0$  gives examples of fairly intricate possibilities for the structure of $A'$ and $B'$. 
 Whether Conjecture \ref{conjecture} is valid (as well as  \eqref{strongconjbound}), or even a somewhat weaker version, is wide open.

\section{An Application to Restricted Sumsets}\label{sec-app}

In this section, we give a simple application of Theorem \ref{new}.
For an abelian group $G$, two sets $A,\, B\subseteq G$, and a map $\tau: A\longrightarrow G$, the restricted sumset $A\overset{\tau}{+} B$ is defined by
\begin{align*}
    A\overset{\tau}{+} B:=\{a+b: a\in A,\ b\in B,\ b\neq \tau(a)\}.
\end{align*}
Lev \cite{Lev} proved the following lower bound on the size of $A\overset{\tau}{+} B$.

\begin{theorem}[Lev \cite{Lev}]
    Let $G$ be a finite abelian group, let $A,\, B\subseteq G$ be subsets with $|A|+|B|\ge |G|+1$, and let $\tau: A\lhook\joinrel\longrightarrow G$ be an injective map. Then
    \begin{align}\label{levbound}
        |A\overset{\tau}{+} B|>|G|-\sqrt{|G|}-\frac{1}{2}.
    \end{align}
\end{theorem}

For an injective map $\tau: A\lhook\joinrel\longrightarrow G$, denote $|\{(a,\ b): a\in A,\ b\in B,\ \tau(a)=b\}|$ by $|\tau|$. We have $|\tau|\le \min\{|A|,\ |B|\}$. Using Theorem \ref{new}, we prove a new lower bound on $|A\overset{\tau}{+} B|$, which is better than \eqref{levbound} when one of $|A|$ and $|B|$ is small, namely, when $\min\{|A|,\ |B|\}<\frac{3}{16}|G|+\frac{5}{16}\sqrt{|G|}+\frac{25}{192}$.

\begin{theorem}
    Let $G$ be a finite abelian group, let $A,\, B\subseteq G$ be  subsets with $|A|+|B|\ge |G|+1$, and let $\tau: A\lhook\joinrel\longrightarrow G$ be an injective map. Then
    \begin{align*}
        |A\overset{\tau}{+} B|\ge |G|+\frac{1-4\sqrt{3\min\{|A|,\ |B|\}}}{3}.
    \end{align*}
\end{theorem}

\begin{proof}
    Let $M:=\min\{|A|,\ |B|\}$.

    If $M=1$, then as $|A|+|B|\ge |G|+1$, we must have $\max\{|A|,\ |B|\}=|G|$ and
    \begin{align*}
        |A\overset{\tau}{+} B|=|G|-1> |G|+\frac{1-4\sqrt{3}}{3}.
    \end{align*}

    Now assume $M\ge 2$. We will show that, for any integer $t\in [2,\, M]$, we always have
    \begin{align}\label{bothcases}
        |A\overset{\tau}{+} B|\ge |G|-\frac{4}{3}t-\frac{M}{t}+\frac{5}{3}.
    \end{align}

\smallskip

    \textbf{Case I.} $\sum_{i=1}^t|A+_iB|\geq t|A|+t|B|-\frac{4}{3}t^2+\frac{2}{3}t$.

    Trivially, we have
    \begin{align*}
        t|A\overset{\tau}{+} B|+|\tau|\ge \sum_{i=1}^t |A+_i B|,
    \end{align*}
    which together with the hypothesis of Case I yields
    \begin{align*}
        |A\overset{\tau}{+} B|&\ge |A|+|B|-\frac{4}{3}t+\frac{2}{3}-\frac{|\tau|}{t} \\
        &\ge |G|-\frac{4}{3}t-\frac{|\tau|}{t}+\frac{5}{3} \\
        &\ge |G|-\frac{4}{3}t-\frac{M}{t}+\frac{5}{3}.
    \end{align*}

\smallskip

    \textbf{Case II.} $\sum_{i=1}^t|A+_iB|< t|A|+t|B|-\frac{4}{3}t^2+\frac{2}{3}t$.

In this case, we can apply Theorem \ref{new} to $A$ and $B$ resulting in subsets $A'\subseteq A$ and $B'\subseteq B$ such that $\ell:=|A\setminus A'|+|B\setminus B'|\leq t-1$ and $A'+B'=A'+_tB'=A+_tB$.  In particular,  for any $a\in A'$ and $b\in B'$ with $\tau(a)=b$, we have  $\mathsf r_{A,\,B}(a+b)\ge t$. As $|A|+|B|\ge |G|+1$, the Pigeonhole Bound (Proposition \ref{pigeon}) implies $A+B=G$. Thus
    \begin{align*}
        |A\overset{\tau}{+} B|\ge& |A+B|-|\{(a,\ b):\; \tau(a)=b\mbox{ and either $a\in A\setminus A'$  or $b\in B\setminus B'$}\}| \\
        &-\frac{1}{t}|\{(a',\ b'): \; \tau(a')=b',\;a'\in A'\und b'\in B'\}| \\
        =& |G|-|\{(a,\ b):\; \tau(a)=b\mbox{ and either $a\in A\setminus A'$  or $b\in B\setminus B'$}\}| \\
        &-\frac{1}{t}(|\tau|-|\{(a,\ b):\; \tau(a)=b\mbox{ and either $a\in A\setminus A'$  or $b\in B\setminus B'$}\}|) \\
        \ge& |G|-\ell-\frac{1}{t}(M-\ell)
        \ge |G|-(t-1)-\frac{1}{t}(M-(t-1)) \\
        =& |G|-t-\frac{M+1}{t}+2.
    \end{align*}
    As $t\ge 2$, we have
        $\bigl(-t-\frac{M+1}{t}+2\bigr)-\bigl(-\frac{4}{3}t-\frac{M}{t}+\frac{5}{3}\bigr)=\frac{t^2+t-3}{3t}>0$.
    So we also have
    \begin{align*}
        |A\overset{\tau}{+} B|\ge |G|-\frac{4}{3}t-\frac{M}{t}+\frac{5}{3}
    \end{align*}
    in this case.

    So, in both cases, we have \eqref{bothcases}, as desired. Note the bound \eqref{bothcases}, considered as a function of a positive real variable $t\in \R^+$, is maximized for $t=\frac{\sqrt{3M}}{2}$.
    Choose $t=\Bigl\lceil\frac{\sqrt{3M}}{2}\Bigr\rceil\in [2,\, M]$. Then
    \begin{align}\label{ineqa}
        -\frac43t=-\frac{4}{3}\cdot\Biggl\lceil\frac{\sqrt{3M}}{2}\Biggr\rceil\ge -\frac{4}{3}\cdot\frac{\sqrt{3M}}{2}-\frac{4}{3},
    \end{align}
    and
    \begin{align}\label{ineqb}
        -\frac{M}{t}=-\frac{M}{\Bigl\lceil\frac{\sqrt{3M}}{2}
        \Bigr\rceil}\ge
         -\frac{M}{\frac{\sqrt{3M}}{2}}.
    \end{align}
    Thus, taking $t=\Bigl\lceil\frac{\sqrt{3M}}{2}\Bigr\rceil$ and combining \eqref{bothcases}, \eqref{ineqa}, and \eqref{ineqb}, we have
    \begin{align*}
        |A\overset{\tau}{+} B|&\ge
        |G| -\frac{4}{3}\cdot\frac{\sqrt{3M}}{2}-\frac{4}{3} -\frac{M}{\frac{\sqrt{3M}}{2}}+\frac53\\
       &=|G|+\frac{1-4\sqrt{3M}}{3}.
    \end{align*}
\end{proof}

\end{document}